\documentclass[dvips,ejs,twoside,preprint]{imsart}

\RequirePackage[OT1]{fontenc}
\RequirePackage{amsthm,amsmath,amssymb,graphicx}
\RequirePackage[numbers]{natbib}
\RequirePackage[colorlinks,citecolor=blue,urlcolor=blue]{hyperref}
\RequirePackage{hypernat}

\doi{10.1214/11-EJS626}
\pubyear{2011}
\volume{5}
\issue{0}
\firstpage{750}
\lastpage{774}

\startlocaldefs
\def\E{\mathbb{E}}
\def\P{\mathbb{P}}

\def\N{\mathbb{N}}
\numberwithin{equation}{section}
\theoremstyle{plain}
\newtheorem{thm}{Theorem}[section]
\newtheorem{prop}{Proposition}[section]

\newtheorem{cor}{Corollary}[section]
\theoremstyle{definition}
\newtheorem{dfn}{Definition}[section]
\newtheorem{rmk}{Remark}[section]
\endlocaldefs

\begin{document}

\begin{frontmatter}
\title{Sparsity considerations for dependent variables}
\runtitle{Sparsity considerations for dependent variables}

\begin{aug}
\author{\fnms{Pierre} \snm{Alquier}\corref{}\thanksref{t1}\ead[label=e1]{alquier@ensae.fr}
\ead[label=u1,url]{http://alquier.ensae.net/}}

\address{CREST - ENSAE \\
3, avenue Pierre Larousse \\
92240 Malakoff France \\
and \\
 LPMA - Universit\'e Paris 7 \\
 175, rue du Chevaleret \\
75205 Paris Cedex 13 France\\
\printead{e1}\\
\printead{u1}}
\end{aug}
\vspace*{-6pt}
\medskip
\textbf{\and}
\vspace*{-6pt}
\begin{aug}
\author{\fnms{Paul} \snm{Doukhan}\ead[label=e2]{doukhan@u-cergy.fr}\ead[label=u2,url]{http://doukhan.u-cergy.fr/}}

\address{Universit\'e de Cergy-Pontoise \\
Laboratoire de Math\'ematiques - Analyse, G'eom\'etrie, Mod\'elisation \\
Site de Saint-Martin \\
2, avenue Adolphe Chauvin \\
95302 Cergy-Pontoise Cedex France \\
\printead{e2}\\
\printead{u2}
}
\end{aug}

\thankstext{t1}{Research partially supported by the French ``Agence
Nationale pour la Recherche''
under grant ANR-09-BLAN-0128 ``PARCIMONIE''.}

\thankstext{t2}{The authors would like to thank the anonymous
reviewers for their valuable comments and suggestions to improve the quality of the paper.}

\runauthor{P. Alquier and P. Doukhan}

\affiliation{Univ. Paris 7, CREST and Univ. Cergy}

\begin{abstract}
The aim of this paper is to provide a comprehensive introduction for the study
of $\ell_{1}$-penalized estimators in the context of dependent observations.
We define a general $\ell_{1}$-penalized estimator for solving
problems of stochastic optimization. This estimator
turns out to be the LASSO \cite{LassoTib} in the regression estimation setting.
Powerful theoretical guarantees on the statistical performances of the LASSO were
provided in recent papers, however, they usually only deal with the iid case.
Here, we study this estimator under various dependence assumptions.
\end{abstract}

\begin{keyword}[class=AMS]
\kwd[Primary ]{62J07}
\kwd[; secondary ]{62M10}
\kwd{62J05}
\kwd{62G07}
\kwd{62G08}
\end{keyword}

\begin{keyword}
\kwd{Estimation in high dimension}
\kwd{weak dependence}
\kwd{sparsity}
\kwd{deviation of empirical mean}
\kwd{penalization}
\kwd{LASSO}
\kwd{regression estimation}
\kwd{density estimation}
\end{keyword}

\received{\smonth{2} \syear{2011}}

\end{frontmatter}

\vspace*{-12pt}
\tableofcontents

\section{Introduction}

\subsection{Sparsity in high dimensional estimation problems}

In the last few years, statistical problems in large dimension received a lot of
attention. That is,
estimation problems where the dimension of the parameter to be estimated, say $p$, is larger than the
size of the sample, usually denoted by $n$. This setting is motivated by modern
applications such as
genomics, where we often have $n\leq 100$ the number of patients with a very
rare desease, and $p$ of
the order of $10^5$ or even $10^6$ (CGH arrays), see for example \cite{vert} and the
references therein.
Other examples appear in econometrics, we refer the reader to Belloni and Chernozhukov
\cite{cherno1,cherno2}.

Probably the most famous example is high dimensional regression estimation: one
observes pairs $(x_{i},y_{i})$
for $1\leq i \leq n$ with $y_{i}\in\mathbb{R}$, $x_{i}\in\mathbb{R}^{p}$ and one
wants to find a
$\theta\in\mathbb{R}^{p}$ such that for a new pair $(x,y)$, $\theta'x$ would be
a good prediction for $y$.
If $p\geq n$, it is well known that a good estimation cannot be performed unless
we make an additional
assumption. Very often, it is quite natural to assume that there exists
such a $\theta$ that is
sparse: most of its coordinates are equal to $0$. If we let $\|\theta\|_{0}$
denote the number of
non-zero coordinates in $\theta$, this means that $\|\theta\|_{0}\ll p$. In the
genomics example,
it means that only a few genes are relevant to explain the desease. Early
examples of estimators introduced to deal with
this kind of problems include the now famous AIC \cite{aic} and BIC \cite{bic}.
Both can be written
\begin{equation}
\label{aicbic}
\arg\min_{\theta\in\mathbb{R}^{p}}  \left\{ \sum_{i=1}^{n} \left(y_i
-\theta'x_i\right)^{2}
         + \lambda_n \|\theta\|_{0} \right\}
\end{equation}
where $\lambda_n>0$ differs in AIC and BIC.
Despite AIC and BIC may give poor results when $p\geq n$ (see \cite{birgemassart}),
taking $\lambda\geq 2\sigma\log(p)$ leads to estimators with very
satisfying statistical
properties ($\sigma^2$ being the variance of the noise). See for
example \cite{birgemassart,BTW07} for such results, and \cite{baraudgiraud} in
the case of unknown variance.

The main problem with this so-called $\ell_0$ penalization approach is that the effective
computation of the estimators defined in \eqref{aicbic} is very time consuming.
In practice, these estimators
cannot be used for $p$ more than a few tens. This motivated the study of the
LASSO introduced by Tibshirani
\cite{LassoTib}. This estimator is defined by
\begin{equation*}
\arg\min_{\theta\in\mathbb{R}^{p}}  \left\{ \sum_{i=1}^{n} \left(y_i
-\theta'x_i\right)^{2}
         + \lambda_n \|\theta\|_{1} \right\}.
\end{equation*}
The convexity of this minimization problem ensures that the estimator can be
computed for very large $p$,
see Efron {\it et al.} \cite{LARS} for example. This motivated a lot of
theoretical studies on the statistical
performances of this estimator. The results with the weakest hypothesis can be
found in the work of Bickel {\it
et al.} \cite{Lasso3} or Koltchinksii \cite{kolt}.
See also very nice reviews in the paper by Van de Geer
and B\"uhlmann \cite{vdgb} or in
the PhD Thesis of Hebiri \cite{hebiri}. Also note that
a quantity of variants of the idea of $\ell_{1}$-penalization were studied
simultaneously to the LASSO: among
others the basis pursuit \cite{basis1,basis2}, the Dantzig Selector
\cite{Dantzig}, the Elastic Net
\cite{zou-elastic}...

Another problem of estimation in high dimension is the so-called problem of
sparse density estimation. In this
setting, we observe $n$ random variables with (unknown) density $f$ and the
purpose is to estimate $f$ as a linear
combination of some functions $\varphi_1$, \ldots, $\varphi_p$. If $p\geq n$ and
$$ f(\cdot)\simeq \sum_{j=1}^{p}\theta_j \varphi_j(\cdot) $$
we can use the SPADES (for SPArse Density Estimator) by Bunea {\it et al.}
\cite{spades,spades2} or the iterative feature
selection procedure in \cite{alquierdensity}.

One of the common features of all the theoretical studies of sparse estimators is
that they focus only on
the case where the observations are independent. For example, for the density
estimation case, in \cite{spades2}
and \cite{alquierdensity} the observations are assumed to be iid. The purpose of
this paper is to propose a
unified framework. Namely, we define a general stochastic optimization problem
that contains as a special case
regression and density estimation. We then define a general $\ell_1$-penalized
estimator for this problem, in
the special case of regression estimation this estimator is actually the LASSO
and in the case of density estimation
it is SPADES. Finally, we provide guarantees on the statistical performances of
this estimator in the spirit of
\cite{Lasso3}, but we do not only consider independent observations: we want to
study the case of dependent
observations, and prove that we can still recover the target $\theta$ in this
case, under various hypothesis.

\subsection{\texorpdfstring{General setting and $\ell_1$-penalized
estimator}{General setting and $ell_1$-penalized estimator}}

We now give the general setting and notations of our paper. Note that the cases
of regression and density
estimation will appear as particular cases.

We observe $n$ random variables in $\mathcal{Z}: Z_{1}, \ldots, Z_{n}$. Let
$\mathbb{P}$ be the distribution
of $(Z_{1},\ldots,Z_{n})$. We have a function
$ Q: \mathcal{Z} \times \mathbb{R}^{p} \rightarrow \mathbb{R} $
such that for any $z\in\mathcal{Z}$, $\theta\in\mathbb{R}^{p}\mapsto
Q(z,\theta)$ is a
quadratic function. The objective is the estimation of a value
$\overline{\theta}$ that minimizes the following expression which only depends
on $n$ and $\theta$:
$$ R(\theta)= \frac{1}{n}\sum_{i=1}^{n}\mathbb{E} Q(Z_{i},\theta) =
\int_{\mathcal{Z}^{n}} \frac{1}{n}\sum_{i=1}^{n} Q(z_{i},\theta)
d\mathbb{P}(z_{1},\ldots,z_{n}) .$$
All the results that will follow are intended to be interesting in the case
$p>n$ on the condition that
$\|\overline{\theta}\|_{0}:={\rm card}\{j:\overline{\theta}_{j} \neq 0\} $ is
small.

We use the following estimator:
$$  \arg\min_{\theta\in\mathbb{R}^{p}} \left[
                         \frac{1}{n}\sum_{i=1}^{n} Q(Z_{i},\theta) + \lambda
\|\theta\|_{1}\right] $$
and $\hat{\theta}_{\lambda}$ denotes any solution of this minimization problem.
\medskip

We now detail the notations in the two examples of interest:
\begin{enumerate}
\item in the regression example, $Z_{i}=(X_{i},Y_{i})$ with the
$X_{i}\in\mathbb{R}^{p}$ deterministic,
and
\begin{equation}
\label{eqreg}
Y_{i}=X_{i}'\theta+\varepsilon_{i}
\end{equation}
where $\mathbb{E}(\varepsilon_{i})=0$ (the $\varepsilon_{i}$
are not necessarily iid, they may be dependent and have different distribution).
Here we
take $Q((x,y),\theta)=(y-x'\theta)^{2}$. In this example,
$\hat{\theta}_{\lambda}$ is known as the
LASSO estimator \cite{LassoTib}.
\item in the density estimation case, $Z_{i}\in\mathbb{R}$ have the same density
wrt Lebesgue measure (but they
are not necessarily independent). We have a family of functions
$(\varphi_{i})_{i=1}^{p}$ and we want
to estimate the density $f$ of $Z_{i}$ by functions of the form
$$ f_{\theta}(\cdot) = \sum_{i=1}^{p} \theta_{i} \varphi_{i}(\cdot) .$$
In this case we take
$$ Q(z,\theta)=\int f_{\theta}^{2}(\zeta)d\zeta - 2 f_{\theta}(z) $$
and note that this leads to
$$ R(\theta) = \int \left(f_{\theta}(x) - f(x)\right)^{2}dx - \int f^{2}(x)dx
                = \int \left(f_{\theta}(x) - f(x)\right)^{2}dx - {\rm cst} .$$
Then $\hat{\theta}_{\lambda}$ is the estimator known as SPADES \cite{spades2}.
\end{enumerate}

\subsection{Overview of the paper}

In Section \ref{sectionmainresult} we provide a sparsity inequality that extend
the one
of Bickel {\it et al.} \cite{Lasso3} to the case of non iid variables. This
result involves
two assumptions: the first one is about the function $Q$ and is already needed in
the iid
  case. It is usually refered as
Restricted Eigenvalue Property.
The other hypothesis is more involved, it is specific to the non iid case. It
roughly says that we are able to control the deviations of empirical means of
dependent variables around their expectations.

In Section \ref{examples}, we provide several examples of classical assumptions on the
observations that can ensure that we have such a control. These assumptions
are expressed in terms of weak dependence coefficients, so in the beginning of this
section we briefly introduce weak dependence. We also provide some references.

We apply
the results of Sections \ref{sectionmainresult} and \ref{examples} to regression
estimation
in Section \ref{regression} and to density estimation in Section \ref{density}.

Finally the proofs are given in Section \ref{proofs}.

\section{Main result}
\label{sectionmainresult}

\subsection{Assumptions and result}

First, we need an assumption on the quadratic form $R(\cdot)$.
\medskip

{\it\noindent {\bf Assumption} ${\bf A}(\kappa)$ with $\kappa>0$. As $Q(z,\cdot)$ is a
quadratic
form, we have the matrix $$\mathbf{M}=\frac{\partial^{2}}{\partial
\theta^{2}} \frac{1}{n}\sum_{i=1}^{n}Q(Z_{i},\theta) $$ that does not depend on $\theta$,
and we assume that the matrix $\mathbf{M}$ has only $1$ on its diagonal
(actually, this just means
that we renormalize the observations $X_{i}$ in the regression case, or the
function
$\varphi_{j}$ in the density estimation case), that it is non-random (here again,
this is easily checked in the two examples) and that it satisfies}
$$
\kappa \leq \inf
\left\{\frac{v'\mathbf{M} v}{\sum_{j\in J}v_{j}^{2}}
 \left|
\begin{array}{l}
v\in\mathbb{R}^{p}, \quad
 J\subset\{1,\ldots,p\}, \quad  |J|<\|\overline{\theta}\|_{0} \\
 \sum_{j\notin J} |v_{j}| \leq 3 \sum_{j\in J} |v_{j}|
\end{array}
\right.\right\}.
$$
Note that this condition, usually referred as restricted eigenvalue property (REP),
is already required in the iid setting, see
\cite{Lasso3,vdgb} for example.
In these paper it is also discussed why we cannot hope to get rid of this
hypothesis.

We set for simplicity
\begin{equation*}
W_i^{(j)}=\frac{1}{2}\frac{\partial
Q(Z_{i},\overline{\theta})}
                      {\partial \theta_{j}}, \qquad i\in\{1,\ldots , n\},\quad
j\in\{1,\ldots , p\}.
\end{equation*}
Recall that as $Q(z,\theta)$ is a quadratic function it may be written as
$Q(z,\theta)=\theta'A(z)\theta+b(z)'\theta+ c(z)$ for a $p\times p$-matrix
valued function $A$ on $\mathbb{R}^p$ and a vector function
$b:\mathbb{R}^p\to\mathbb{R}^p$ so that
$$W_i^{(j)}=(A(Z_i)\overline\theta)_{j}+\frac12(b(Z_i))_j .$$
\begin{thm}
\label{mainresult}
Let us assume that Assumption ${\bf A}(\kappa)$ is satisfied. Let us assume
that the distribution $\mathbb{P}$ of $(Z_{1},\ldots,Z_{n})$ is such that there
is a constant
$\alpha\in[0,\frac{1}{2}]$ and a decreasing continuous function $\psi(\cdot)$
with
\begin{equation}
\label{conditionthm}
\forall j\in\{1,\ldots,p\},\quad \mathbb{P}
    \left(\left|\frac{1}{n}\sum_{i=1}^{n}W_i^{(j)}\right|\geq
n^{-\frac{1}{2}+\alpha} t \right) \leq \psi(t) .
\end{equation}
Let us put
$$ \lambda \geq \lambda^{*}:= 4 n^{\alpha-\frac{1}{2}}
\psi^{-1}\left(\frac{\varepsilon}{p}\right) .$$
Then
$$
\mathbb{P}\left\{
\begin{array}{c}
\displaystyle{R(\hat{\theta}_{\lambda})-R(\overline{\theta})
                   \leq \frac{4\lambda^{2} \|\overline{\theta}\|_{0}}{\kappa}}
\\
\\
{\rm and,} \quad {\rm simultaneously}
\\
\\
\displaystyle{\|\hat{\theta}_{\lambda}-\overline{\theta}\|_{1}
     \leq \frac{2\lambda \|\overline{\theta}\|_{0}}{\kappa}}
\end{array}
\right\}\geq 1-\varepsilon.
$$
\end{thm}
The arguments of the proof of Theorem \ref{mainresult} are taken from
\cite{Lasso3}.
The proof is given in Section \ref{proofs}, page \pageref{proofs}.

Note that the hypothesis in this theorem heavily depend on the distribution of
the variables $Z_{1}$, \ldots, $Z_{n}$, and particulary on their type of
dependence. Section \ref{examples} will provide some examples of situations
where this hypothesis is satisfied.

Also note that the upper bound in the inequality is minimized if we make the
choice $ \lambda = \lambda^{*}$. Then
$$
\mathbb{P}\left\{
\begin{array}{c}
\displaystyle{R(\hat{\theta}_{\lambda})-R(\overline{\theta})
                   \leq \frac{64 }{\kappa}
\frac{\|\overline{\theta}\|_{0}\left[\psi^{-1}\left({\varepsilon}/p\right)\right
]^{2}}{n^{1-2\alpha}}}
\\
\\
{\rm and}
\\
\\
\displaystyle{\|\hat{\theta}_{\lambda}-\overline{\theta}\|_{1}
     \leq \frac{8}{\kappa}
\frac{\|\overline{\theta}\|_{0}\left[\psi^{-1}\left({\varepsilon}/p\right)\right
]^{2}}{n^{\frac{1}{2}-\alpha}}}
\end{array}
\right\}\geq 1-\varepsilon.
$$
It is important to remark that the choice $\lambda=4 n^{\alpha-\frac{1}{2}}
\psi^{-1}\bigl(\frac{\varepsilon}{p}\bigr)$ may be impossible
in practice, as the practitionner may not know $\alpha$ and $\psi(\cdot)$.
Moreover, this choice is not necessarily
the best one in practice: in the regression case with iid noise
$\mathcal{N}(0,\sigma^{2})$, we will see that this
choice leads to $\lambda=4\sigma\sqrt{2n\log(p/\varepsilon)}$. This choice requires
the knowledge of $\sigma$. Moreover it is
not usually the best choice in practice, see for
example the simulations in \cite{hebiri}. Even in the iid case, the choice of a
good $\lambda$ in practice is still
an open problem. However, note that
\begin{enumerate}
\item the question is in some sense meaningless. For example the value of
$\lambda$ that minimizes
the quadratic risk $R(\hat{\theta}_{\lambda})$ is not the same than the value of
$\lambda$ that may ensure, under
some supplementary hypothesis, that $\hat{\theta}_{\lambda}$ identifies
correctly the non-zero coordinates in
$\overline{\theta}$, see for example Leeb and P\"otscher \cite{leeb} on that
topic. One has to be careful to
what one means when one say {\it a good choice for} $\lambda$.
\item some popular methods like cross-validation seem to give good results for
the quadratic risk, at least in
the iid case. An interesting open question is to know if one can prove
theoretical results for cross validation in this setting. See also the bootstrap
method proposed in \cite{cherno2}.
\item the LARS algorithm \cite{LARS} compute $\hat{\theta}_{\lambda}$ for any
$\lambda>0$ in a very short time (coordinate descent algorithms \cite{PCO}
are valuable alternative to LARS).
\end{enumerate}

\subsection{Remarks on the density and regression estimation setting}

First, note that in the regression setting (Equation \ref{eqreg}), for any $i\in\{1,\ldots,n\}$
and $j\in\{1,\ldots,p\}$ we have
\begin{equation*}
W_i^{(j)} = (X_{j})_{i}(Y_{i}-X_{i}'\overline{\theta}) = (X_{j})_{i}\varepsilon_{i}.
\end{equation*}
Then, in the density estimation context,
\begin{multline*}
W^{(j)}_{i}
     = \int \varphi_{j}(x) f_{\overline{\theta}}(x)
dx- \varphi_{j}(Z_{i}) = \int \varphi_{j}(x) f(x)
dx-\varphi_{j}(Z_{i})
\\ =
\mathbb{E}[\varphi_{j}(Z_{1})]-\varphi_{j}(Z_{i}).
\end{multline*}

So, in both cases, the assumption given by Equation \ref{conditionthm} is
satisfied if
we have a control of the deviation of empirical means to their expectation. In
the next
sections, we discuss some conditions to obtain such controls with dependent
variables.

\section{Models fitting conditions of Theorem \ref{mainresult}}
\label{examples}

In this section, we give some results that allow to control the deviation
of empirical means to their expectations for general (non iid) obsrevations.
The idea will be, in the next sections, to apply these results to the
processes $W^{(j)}=(W^{(j)}_{i})_{1\leq i \leq n}$ for $1\leq j \leq p$.
For the sake of simplicity, in this section, we deal with a generic process
$V=(V_i)_{i\in\mathbb{Z}}$ and the applications are given in the next sections.
Various examples of pairs $(\alpha,\psi)$ are given. We will use the classical
notation
$$ S_n = \sum_{i=1}^{n} V_i .$$

\subsection{\texorpdfstring{Weak dependence ($\alpha=0$)}{Weak dependence ($alpha=0$)}}

We are going to introduce some coefficients in order to control the dependence
of the $V_i$. The first example of such coefficients are the
$\alpha$-mixing coefficients first introduced by Rosenblatt \cite{alphaR},
$$ \alpha_{V}(r) = \sup_{t\in\mathbb{Z}}
           \sup_{
\tiny{
\begin{array}{c}
U\in\sigma(V_{i},i\leq t) \\
U'\in\sigma(V_{i},i\geq t+r)
\end{array}
}
} \left| \mathbb{P}(U\cap U')
            - \mathbb{P}(U)\mathbb{P}(U') \right| .$$
The idea is that the faster $\alpha_{V}(r)$ decreases to $0$, the less dependent
are $V_i$ and $V_{i+r}$ for large $r$. Assumptions on the rate of decay allows
to prove laws of large numbers
and central limit theorems. Different mixing coefficients were then
studied, we refer the reader to \cite{MR1312160,Rio00} for more details.

The main problem with mixing coefficients is that they exclude too many processes.
It is easy to build a process $V$ satisfying a central limit theorem
with constant $\alpha_{V}(r)$, see \cite{MR2338725} Chapter 1 for an example. This motivated
the introduction of {\it weak dependence} coefficients. The monograph \cite{MR2338725}
provides a comprehensive introduction to the various weak dependence coefficients.
Our purpose here is not to define all these coefficients, but rather to introduce
some examples that allow to satisfy condition~\eqref{conditionthm} in
Theorem \ref{mainresult}.
\begin{dfn} We put, for any process $(V_{i})_{i\in\mathbb{Z}}$,
\begin{eqnarray}\label{coef}
c_{V,m}(r)  =\max_{1\le \ell<m}\!\!\!\sup_{{\tiny \begin{array}{c} t_1\le
\cdots \le t_m
\\ t_{\ell+1}-t_{\ell}\ge r\end{array}}}\!\!\!
\left|{\rm cov }\left(V_{t_1}\cdots V_{t_\ell},V_{t_{\ell+1}}\cdots
V_{t_m}\right)\right|.
\end{eqnarray}
\end{dfn}
We precise in \S-\ref{sectmom} and
in \S-\ref{sectexp} that suitable decays of those coefficients yield~\eqref{conditionthm}.
Those two sections will  provide quite different forms of the function~$\psi$.

\begin{dfn}
Let us assume that for any $r\geq 0$, for any
$g_{1}$ and $g_{2}$ respectively $L_1$ and $L_2$-Lipschitz,
where eg.,
$$
L_1 := \sup_{(x_1,\ldots,x_\ell)\ne (y_1,\ldots,y_\ell)}
\frac
{g_1(y_1,\ldots,y_\ell)-g_1(x_1,\ldots,x_\ell)}
{|y_1-x_1|+\cdots+|y_\ell-x_\ell|}.
$$
We also assume that
for any $t_1\leq \cdots \leq t_\ell \leq t_{\ell+1} \leq \cdots \leq t_{m}$
with $t_{\ell+1}-t_{\ell}\geq r$,
$$
\left|{\rm cov }\left[g_{1}(V_{t_1},\ldots, V_{t_\ell}),g_{2}(V_{t_{\ell+1}},\ldots,
V_{t_m})\right]\right| \leq \psi(L_1,L_2,\ell,m-\ell)\eta_{V}(r)
$$
with $\psi(L_1,L_2,\ell,\ell')=\ell L_1 + \ell' L_2$.
Then $V$ is said to be $\eta$-dependent with $\eta$-dependence coefficients $(\eta(r),r\geq 0)$.
\end{dfn}
\begin{rmk}
Other functions $\psi(L_1,L_2,\ell,\ell')$ allow to define the
$\lambda$, $\kappa$ and $\zeta$-dependence, see \cite{MR2338725}.
\end{rmk}
We finally provide some basic properties, proved in \cite{MR2338725}.
The following result allows a comparison between different type of coefficients.
\begin{prop}
If $\sup_{i} \|V_i\|_{\infty}\leq M$ then
\begin{eqnarray}
\label{ineqcoef}
c_{V,m}(r)&\le& mM^m\eta_{V}(r)\\
\nonumber &\le& M^m\alpha_{V}(r).
\end{eqnarray}
\end{prop}
Finally the following property will be useful in this paper.
\begin{prop}
\label{lipeta}
If $V$ is $\eta$-dependent and $f$ is $L$-Lipschitz and bounded,
then $f(V)$ is also $\eta$-dependent with
$$ \eta_{f(V)}(r) = L \eta_{V}(r) .$$
\end{prop}

\subsubsection{Moment inequalities}\label{sectmom}

In Doukhan and Louhichi \cite{MR1719345} it is proved that if for an
even integer $2q$ we have
\begin{equation}\label{borneC}
\exists C\ge1\ \mbox{ such that:}\qquad c_{V,2q}(r)\le C(r+1)^{-q},
\quad\forall  r\ge0
\end{equation}
then Marcinkiewicz-Zygmund inequality follows:
$$
\mathbb{E}\left((V_1+\cdots+V_n)^{2q}\right)={\cal O}(n^{q})
$$
and thus $\alpha=0$ and $\psi(t)$ is of the order of $1/t^{2q}$ in~\eqref{conditionthm}.
However,
explicit constants are needed in Theorem \ref{mainresult}.
We actually have the following result.
\begin{prop}\label{momcomb}
Assume that coefficients (\ref{coef}) fit the relation (\ref{borneC}) for some
integer $q\ge1$, then Marcinkiewicz-Zygmund inequality follows
\begin{equation*}
\E\left[(V_1+\cdots+V_n)^{2q} \right]\le C^qd_{2q}(2q)!n^q
\end{equation*}
where
$$ d_{m} \equiv \frac1m\frac{(2m-2)!}{((m-1)!)^2}, \qquad m=2, 3, \ldots $$
\end{prop}
The proof follows \cite{MR1719345}, it is given in Section \ref{proofs}.
\begin{rmk}
Sharper constants $a_{2q}$ are also derived in the proof (Equation \eqref{eq7}, page
\pageref{eq7}), one may replace the constants
$2d_2,24d_4,720d_6$ by 1, 4 and 17 and using the recursion (\ref{recur}) also
improves the above mentioned bounds.
\end{rmk}
Various inequalities of this type where derived for alternative dependences (see
Doukhan
\cite{MR1312160}, Rio \cite{Rio00} and Dedecker \textit{et al.}
\cite{MR2338725} for an extensive bibliography
which also covers the case of non integer exponents).

\subsubsection{Exponential inequalities}\label{sectexp}

Using the previous inequality, Doukhan and Louhichi \cite{MR1719345} proved
exponential inequalities that would lead to $\psi(t)$ in $\exp(-\sqrt{t})$.
Doukhan and Neumann \cite{MR2330724} use alternative cumulant techniques
to get $\psi(t)$ in $\exp(-t^2)$ for suitable bounds of the previous covariances
(\ref{coef}).
\begin{thm} \label{T2.1} \cite{MR2330724}
Let us assume that $\sup_{i} \|V_i\|_{\infty}\leq M$.
Let $\Psi:\N^{2}\rightarrow\N$
be one of the following functions:
\begin{itemize}
\item[(a)] $\quad \Psi(u,v)=2v$, \item[(b)] $\quad \Psi(u,v)=u+v$,
\item[(c)] $\quad \Psi(u,v)=uv$, \item[(d)] $\quad
\Psi(u,v)=\alpha(u+v)\,+\,(1-\alpha)uv$, \qquad for some
$\alpha\in(0,1)$.
\end{itemize}
We assume that there exist constants $K,L_1,L_2<\infty$, $\mu\geq
0$, and a nonincreasing sequence of real coefficients
$(\rho(n))_{n\geq 0}$ such that, for all $u$-tuples
$(s_1,\ldots,s_u)$ and all $v$-tuples $(t_1,\ldots,t_v)$ with
$\;1\leq s_1\leq\cdots\leq s_u\leq t_1\leq\cdots\leq t_v\leq n\;$
the following inequality is fulfilled:
\begin{equation}
\label{2.1} \left| \mbox{cov }\left( V_{s_1}\cdots V_{s_u}, V_{t_1}\cdots
V_{t_v} \right) \right| \leq K^{2} M^{u+v-2}  \Psi(u,v)
\rho(t_1-s_u),
\end{equation}
where
\begin{equation*}
\sum_{s=0}^{\infty} (s+1)^{k} \rho(s) \leq\, L_1  L_2^k
(k!)^{\mu} \qquad \forall k\geq 0.
\end{equation*}
Then
\begin{equation*}
{\P}\left( S_n \,\geq\, t \right)
\,\leq\, \exp\left( -\frac{t^{2}/2}{A_n \,+\, B_n^{1/(\mu+2)}
t^{(2\mu+3)/(\mu+2)}} \right),
\end{equation*}
where $A_n$ can be chosen as any number greater than or equal to
$\sigma_n^{2}:=Var(V_{1}+\dots+V_{n})$ and
\begin{eqnarray*} B_n = 2  (K\vee M)
L_2 \; \left( \Big( \frac{ 2^{4+\mu} n  K^{2}  L_1 }{ A_n }
\Big) \vee 1 \right).
\end{eqnarray*}
\end{thm}

\begin{rmk}
\label{remarqueDN}
One can easily check that if $V$ is $\eta$-dependent then
\eqref{2.1} is satisfied with $\Psi$ as in {\it (b)}, $K^{2}=M$
and $\rho(r)=\eta(r)$, see Remark 9 page 9 in \cite{MR2330724}. So if $V$ is
$\eta$-dependent and $\eta(r)$ decreases fast enough to $0$ then
we have an exponential inequality.
\end{rmk}
This result yields convienient bounds for the function $\psi$.
A recent paper by Olivier Wintenberger \cite{oliv1} is also of
interest: it directly yields alternative results from our main
result. In this paper, we do not intend to provide the reader
with encyclopedic references but mainly to precise some ideas and
techniques so that this will  be developed in further papers.
\vfill\eject

\subsection{\texorpdfstring{Long range dependence
$(\alpha\in]0,\frac12[$)}{Long range dependence $(alpha\in]0,\frac12[$)}}
\subsubsection{Power decays} Assume now that $V$ is a centered series satisfies
$\sum_i\sup_k\left|{\rm cov}(V_k,V_{k+i})\right|=\infty$ then
$\alpha>0$
 may occur, eg. if $$r(i)\equiv\sup_k\left|{\rm
cov}(V_{k},V_{k+i})\right|\sim i^{-\beta}$$ for $\beta\in]0,1]$ then
 ${\rm var}(S_n)\sim n^{2-\beta}$; then $\alpha=(1-\beta)/2$ holds.
 \subsubsection{Gaussian case} In the special case of Gaussian processes
$(V_i)_i$, tails of $S_n$ are classically described because
 $S_n\sim{\cal N}(0,\sigma_n^2)$
 and here  $\psi(t)=\exp(-t^2)$. We thus may obtain simultaneously subGaussian
tails and
 $\alpha=(1-\beta)/2>0$.
\subsubsection{Non subGaussian tails}
Assume that that for each $i,j$, $G_i\sim {\cal N}(0,1)$ and
$(G_i)_i$ is a stationary Gaussian processes with, for some
$B$, $\beta$,
\begin{equation}
\label{definitB}
r(i)={\rm cov}(G_{k},G_{k+i})\sim B i^{-\beta} .
\end{equation}
 Let $V_i=P(G_i)$ for a function with Hermite rank $m\ge1$, and
since
$${\rm cov}(H_m(G_0),H_m(G_i))=m!\left(r(i)\right)^m$$
their  covariance series is  non $m$-th summable in case $\beta\in]\frac1m,1[$.

  The case $P(x)=x^2- 1$ and $\beta\in]\frac12,1[$ is investigated by using the
following expansion in the seminal work by Rosenblatt \cite{Ros61}.

 Set $R_n$ for the   covariance matrix of the Gaussian random vector
$(G_1,\ldots,G_n)$:
\begin{eqnarray*}
\E e^{tn^{\beta-1}S_n}&=&
e^{-tn^\beta}\mbox{det}^{-\frac12}
\left(
I_n-2tn^{\beta-1}R_n\right)
\\
&=&\exp\left(
\frac12\sum_{k=2}^\infty\frac1k(2tn^{\beta-1})^k\mbox{ trace
}\left(R_n\right)^k\right).
\end{eqnarray*}
Quoting that
$$n^{k(\beta-1)}\mbox{ trace }(R_n)^k\to_{n\to\infty}
c_k>0
$$
with
$$
c_k=B^k\int_{0}^1\cdots\int_{0}^1
|x_1-x_2|^{-\beta}|x_2-x_3|^{-\beta}\cdots
|x_{k-1}-x_k|^{-\beta}|x_k-x_1|^{-\beta}dx_1\cdots dx_k
$$
($B$ is given by Equation~\eqref{definitB}),
this is thus clear that for small enough
$|t|<\tau=\frac12\sup_{k\ge2,j\ge1}\left(c_k\right)^{\frac1k}$,
$$
\E e^{tn^{\beta-1}S^{(j)}_n}
 \to_{n\to\infty}\exp\left(\frac12\sum_{k=2}^\infty(2t)^k\frac{c_k}
k\right).
$$
Here the conditions in the main theorem hold with $\psi(t)=e^{-t}$ and
$\alpha=\frac12-\beta>0$ for any  $M>1/\tau$.

\section{Application to regression estimation}
\label{regression}

In this section we apply Theorem \ref{mainresult} and the various examples of
Section
\ref{examples} to obtain results for regression estimation. Note that
the results in the iid setting are already known, they are only given here for
the sake
of completeness, in order to provide comparison with the other cases.

Let us remind that in the regression case, we want to apply the results of
Section \ref{examples} to
$$ W_{i}^{(j)} = (X_{j})_{i}\varepsilon_{i} .$$
For the sake of simplicity, in this whole session dedicated to regression,
let us put
$$ \max(X) := \max_{1\leq i \leq n} \max_{1 \leq j \leq p} |(X_{i})_{j}| .$$

\subsection{Regression in the iid case}

Under the usual assumption that the $\varepsilon_{i}$ are iid and subGaussian,
$$ \forall s,\quad \mathbb{E}[\exp(s\varepsilon_{i}^{2})] \leq
\exp\left(\frac{s^{2}\sigma^{2}}{2}\right)$$
for some known $\sigma^{2}$, then we have
$$
  \mathbb{P} \left( \left|\frac{1}{n}\sum_{i=1}^{n}W_i^{(j)}\right|
        \geq \frac{t}{\sqrt{n}} \right)   \leq \psi(t) = \exp(-\frac{t^{2}}{2\sigma^{2}}) .
$$
So we can apply Theorem \ref{mainresult} in order  to obtain the following well
known result:

\begin{cor}[\cite{Lasso3}]
In the context of Equation \ref{eqreg}, under Assumption ${\bf A}(\kappa)$,
if the $(\varepsilon_{i})$ are iid and subGaussian with
variance upper bounded by $\sigma^{2}$,
the choice $\lambda=4\sigma\sqrt{2\log(p/\varepsilon)/n}$ leads to
$$
\mathbb{P}\left(R(\hat{\theta}_{\lambda})-R(\overline{\theta})
                   \leq \frac{128 \sigma^{2}}{\kappa}
                      \frac{\|\overline{\theta}\|_{0}\log\frac{p}{\varepsilon}}
                               {n} \right)\geq 1-\varepsilon.
$$
\end{cor}

\subsection{Regression estimation in the dependent case}

\subsubsection{Marcinkiewicz-Zygmund type inequalities}

Let us remark that, for any $1\leq j \leq p$,
$$ c_{W^{(j)},m}(r) \leq c_{\varepsilon,m}(r)
\left(\max_{i,j}|(X_j)_i|\right)^{m} = \max(X)^{m} c_{\varepsilon,m}(r). $$
Thus, we apply Theorem \ref{mainresult} and Proposition \ref{momcomb} to obtain
the
following result.

\begin{cor}
In the context of Equation \ref{eqreg}, under Assumption ${\bf A}(\kappa)$,
if the $(\varepsilon_{i})$ satisfy, for some even integer $2q$,
\begin{equation*}
\exists C\ge1\ \mbox{ such that:}\qquad\forall  r\ge0,\qquad c_{\varepsilon,2q}(r)\le
C(r+1)^{-q},
\end{equation*}
the choice
$$\lambda= \frac{4C^{\frac{1}{2}}\max(X)^{q}}{\sqrt{n}}
\left(\frac{d_{2q}q!p}{\varepsilon}\right)^{\frac{1}{2q}} $$
leads to
$$
\mathbb{P}\left(R(\hat{\theta}_{\lambda})-R(\overline{\theta})
                   \leq
 \frac{64 C \max(X)^{2q} (d_{2q}q!)^{\frac{1}{q}}}{\kappa}
\frac{\|\overline{\theta}\|_{0} p^{\frac{1}{q}}}{\varepsilon^{\frac{1}{q}} n}
\right)\geq 1-\varepsilon.
$$
\end{cor}

\begin{rmk}
This result aims at filling a gap for non subGaussian and non iid random variables.

The result still allows to deal with the \textsl{sparse} case $p>n$ in case
$q>1$.
In this case we deal with the case $p=n^{q/2}$ and we get a rate of convergence
in probability ${\cal O}(1/\sqrt n)$.

If $q=1$ and $\frac pn\to0$ the least squares methods apply which make such
sparsity algorithms less relevant.

 Moreover if $q<1$ the present method is
definitely not efficient. Hence in the case of heavy tails, such as considered in
the paper by Bartkiewicz {\it et al.} \cite{oliv2}, our results are
useless. Anyway, using least squares for heavy tailed models (without second
order moments) does not look to be a good idea!
\end{rmk}

\subsubsection{Exponential inequalities}

Using Theorem \ref{mainresult} and Theorem \ref{T2.1} we prove the following
result.

\begin{cor}
Let us assume that the $(\varepsilon_{i})$ satisfy the hypothesis of Theorem
\ref{T2.1}: let $\Psi:\N^{2}\rightarrow\N$
be one of the functions of Theorem \ref{T2.1}, we assume that there
are constants $K,L_1,L_2<\infty$, $\mu\geq
0$, and a nonincreasing sequence of real coefficients
$(\rho(n))_{n\geq 0}$ such that, for all $u$-tuples
$(s_1,\ldots,s_u)$ and all $v$-tuples $(t_1,\ldots,t_v)$ with
$\;1\leq s_1\leq\cdots\leq s_u\leq t_1\leq\cdots\leq t_v\leq n\;$
the following inequality is fulfilled:
\begin{equation*}
\left| \mbox{cov }\left( \varepsilon_{s_1}\cdots \varepsilon_{s_u}, \varepsilon_{t_1}\cdots
\varepsilon_{t_v} \right) \right| \leq K^{2} M^{u+v-2}  \Psi(u,v)
\rho(t_1-s_u),
\end{equation*}
where
\begin{equation*}
\sum_{s=0}^{\infty} (s+1)^{k} \rho(s) \leq\, L_1  L_2^k
(k!)^{\mu} \qquad \forall k\geq 0.
\end{equation*}
Let $c$ be a positive constant and let us put
$$ \mathcal{C} := 4 K^{2} \max(X)^{2} \Psi(1,1)L_1 +
c 2L_{2} \max(X) (K\vee M)\left(\frac{2^{\mu+3}}{\Psi(1,1)}\vee 1\right) .$$
Let us assume that $\varepsilon>0$, $p$ and $n$ are such that
$$
p \leq \frac{\varepsilon}{2} \exp\left(\frac{c^{2}
n^{\frac{1}{\mu+2}}}{\mathcal{C}}\right)
$$
then for
$$\lambda =
        4 \sqrt{\frac{\mathcal{C} \log\left(\frac{2 p}{\varepsilon}\right)}{n}} $$
we have
$$
\mathbb{P}\left\{
\displaystyle{R(\hat{\theta}_{\lambda})-R(\overline{\theta})
                   \leq \frac{64 \mathcal{C} }{\kappa}
\frac{\|\overline{\theta}\|_{0} \log \left(\frac{2p}{\varepsilon}\right) }{n}}
\right\}\geq 1-\varepsilon.
$$
\end{cor}

So the rate is the same than in the iid case. The only difference is in the
constant, and a restriction for very large values of $p$.

\begin{proof}
For the sake of shorteness, let us put
$$
 C_{1}  =4 K^{2} \max(X)^{2}\Psi(1,1)L_1 \text{ and }
 C_{2} = 2L_{2}\max(X)(K\vee M)\left(\frac{2^{\mu+3}}{\Psi(1,1)}\vee 1\right)
$$
and note that $\mathcal{C}=C_{1}+cC_{2}$.
First, note that for any $j\in\{1,\ldots,p\}$,
\begin{multline*} \left| \mbox{cov }\left( W^{(j)}_{s_1}\cdots W^{(j)}_{s_u},
W^{(j)}_{t_1}\cdots
W^{(j)}_{t_v} \right) \right| \\
\leq \left(\sup_{i} X_{i}^{(j)}\right)^{u+v} K^{2} M^{u+v-2}  \Psi(u,v)
\rho(t_1-s_u)
\\
\leq \tilde{K}^{2} \tilde{M}^{u+v-2}  \Psi(u,v)
\rho(t_1-s_u)
\end{multline*}
if we put $\tilde{K} = \max(X) K$ and $\tilde{M}=\max(X) M$.
Using Theorem \ref{T2.1}, we obtain for any $j$,
$$
\mathbb{P} \left(\left|\sum_{i} W_{i}^{(j)}\right| \geq t\right)
\leq 2 \exp\left( -\frac{t^{2}/2}{A_n \,+\, B_n^{\frac1{\mu+2}}
t^{\frac{2\mu+3}{\mu+2}}}\right)
$$
where $A_n=\sigma^{2}_{n} \leq 2n\tilde{K}^{2}\Psi(1,1) L_{1}$ and
\begin{eqnarray*} B_n = 2  (K\vee M)
L_2 \; \left( \Big( \frac{ 2^{4+\mu} n  K^{2}  L_1 }{ A_n }
\Big) \vee 1 \right),
\end{eqnarray*}
in other words:
$$
\mathbb{P} \left( \left|\sum_{i} W_{i}^{(j)}\right| \geq t\right)
\leq 2 \exp\left( -\frac{t^{2}/2}{C_{1} n  \,+\, C_{2}
t^{\frac{2\mu+3}{\mu+2}}}\right).
$$
Now, let us put $u=t/\sqrt{n}$, we obtain
\begin{multline*}
\mathbb{P} \left( \left|\frac{1}{n}\sum_{i} W_{i}^{(j)}\right| \geq u n^{-\frac{1}{2}}\right)
\leq 2 \exp\left( -\frac{n u^{2}/2}{C_{1} n  \,+\, C_{2}
n^{\frac{2\mu+3}{2\mu+4}}u^{\frac{2\mu+3}{\mu+2}}}\right)
\\
\leq 2  \exp\left(-\frac{u^{2}/2 }{C_{1} +C_{2} n^{-\frac1{2\mu+4}} u^{\frac{2\mu+3}{\mu+2}} }\right).
\end{multline*}
Remark that we cannot in general compute explicitely the inverse of this function
but we can upper-bound
the range for $u$:
$$
 u \leq c\cdot n^{\frac1{2\mu+4}}
$$
In this case,
$$
\mathbb{P} \left( \left|\frac{1}{n}\sum_{i} W_{i}^{(j)}\right| \geq u n^{-\frac{1}{2}}\right)
\leq 2 \exp\left(-\frac{u^{2}/2 }{C_{1} +C_{2} c }\right)
 = 2 \exp\left(-\frac{u^{2} }{2\,\mathcal{C}} \right) =: \psi(u)
$$
and so
$$ \psi^{-1}(y) = \sqrt{\mathcal{C}\log \left(\frac{2}{y}\right)} .$$
So we can take, following Theorem \ref{mainresult},
$$\lambda = 4 n^{-\frac{1}{2}} \psi^{-1}\left(\frac{\varepsilon}{p}\right) =
        4 n^{-\frac{1}{2}} \sqrt{\mathcal{C} \log\left(\frac{2 p}{\varepsilon}\right)}$$
as soon as $\psi^{-1}(\varepsilon/p)<n^{1/(2\mu+4)}$. For example, for a fixed number of observations
$n$ and a fixed confidence level $\varepsilon$, we have the restriction:
$$
p \leq \frac{\varepsilon}{2}
\exp\left(\frac{c n^{\frac{1}{\mu+2}}}{\mathcal{C}}\right).
$$
Under this condition we have, by Theorem \ref{mainresult},
$$
\mathbb{P}\left\{
\begin{array}{c}
\displaystyle{R(\hat{\theta}_{\lambda})-R(\overline{\theta})
                   \leq \frac{64 \,\mathcal{C} }{\kappa}

\frac{\|\overline{\theta}\|_{0} \log \left(\frac{2p}{\varepsilon}\right) }{n}}
\\
\\
{\rm and,}
\\
\\
\displaystyle{\|\hat{\theta}_{\lambda}-\overline{\theta}\|_{1}
     \leq \frac{8\, \mathcal{C}}{\kappa}

\frac{\|\overline{\theta}\|_{0} \log \left(\frac{2p}{\varepsilon}\right) }{n^{\frac{1}{2}}}}
\end{array}
\right\}\geq 1-\varepsilon,
$$
this ends the proof.
\end{proof}

\subsection{Simulations}

In order to illustrate the results, we propose a very short simulation study.
The purpose of this study is not to show the good performances of the estimator
in practice or to give recipes for the choice of $\lambda$. The aim is more to
show that the performances of the iid setting are likely to be obtained in the
dependent setting if the dependence coefficients are small.

We use the following model:
$$ Y_{i} = \theta' X_{i} + \varepsilon_{i}, \quad 1\leq i \leq n=30 $$
where the $X_{i}$'s will be treated as fixed design, but in practice will
be iid vectors in $\mathbb{R}^{p}$ with $p=50$, with distribution $\mathcal{N}_p(0,\Sigma)$
where $\Sigma$ is given by $\Sigma_{i,j}=0.5^{|i-j|}$.
%
\begin{figure}[!b]
\centering
\includegraphics*[width=11cm,height=6.5cm]{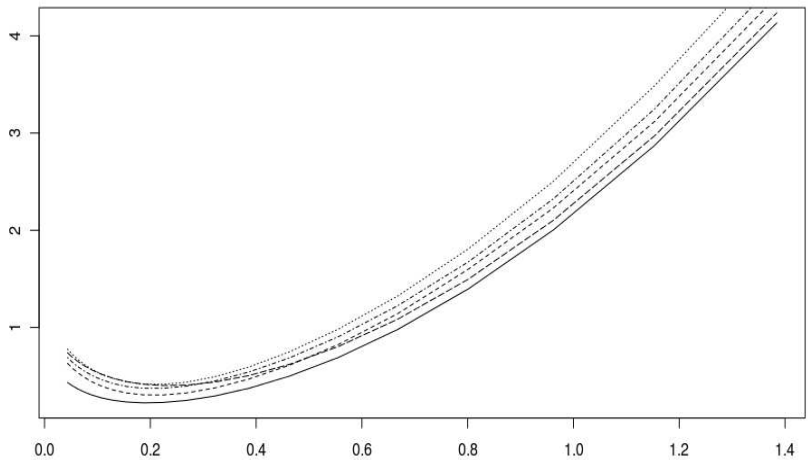}
\caption{\label{courbes}Results of the experiments. The $x$-axis gives the value $g$ where
$\lambda = g\sqrt{\log(p)/n}$. The $y$-axis gives $\sum_{i=1}^{n}(\hat{\theta}_{\lambda}'X_{i}-
\theta'X_{i})^{2}$ the error of reconstruction of the signal. The lines code is the following:
$\vartheta=-0.95$: solid line, $\vartheta=-0.5$: short dashed line, $\vartheta=0$: dotted line,
$\vartheta=0.5$: dot/dash, $\vartheta=0.95$: long dash.}
\end{figure}
The parameter is given
by $\theta=(3,1.5,0,0,2,0,0,\ldots)'\in\mathbb{R}^{p}$. This is the toy example used
by Tibshirani \cite{LassoTib}. Let $\vartheta\in]-1,1[$.

The noise satisfies
$ \varepsilon_{i} = \vartheta \varepsilon_{i-1} + \eta_{i} $, for $i\geq 2$,
where the $\eta_{i}$ are iid $\mathcal{N}(0,1-\vartheta^{2})$
and $\varepsilon_{1}\sim\mathcal{N}(0,1)$. Note that this ensure that
$\mathbb{E}(\varepsilon_{i}^{2})=1$ for any $i$, so the noise level does
not depent on $\vartheta$. In the experiments, $$\vartheta\in\{-0.95,-0.5,
0,0.5,0.95\}.$$
We fixed a grid of values $\mathcal{G}\subset ]0,1.5[$ and we computed, for every
experiment, the LASSO estimator with $\lambda=g \sqrt{\log(p)/n}$ for all $g\in\mathcal{G}$.
We have repeated the experiment $25$ times for every value of $\vartheta$ and report the results
in Figure \ref{courbes}.

We can remark that all the curves are very similar. The minimum reconstruction error is obtained
for $g\simeq 0.2$, that corresponds to $\lambda\simeq 0.072$.
Note that in the iid case, it is smaller than the theoretical value given by
Theorem \ref{mainresult}, $\lambda=4\sigma\sqrt{2\log (p/\varepsilon)/n}\simeq 2.56$
for $\varepsilon=1/10$, that would correspond to $g \simeq 7.10$, a value that would
not event stand in the figure!

\section{Application to density estimation}
\label{density}

Here we apply Theorem \ref{mainresult} and
Section
\ref{examples} to the context of density estimation.
Let us remind that in this setting,
$$ W_{i}^{(j)} =  \mathbb{E}[\varphi_{j}(Z_{1})]-\varphi_{j}(Z_{i}).$$

\subsection{Density estimation in the iid case}

If the $Z_{i}$ are iid with density $f$ and if $\|\varphi_{j}\|_{\infty}<B$ for
any $j\in\{1,\ldots,p\}$ then we can
apply Hoeffding inequality \cite{hoeffding} to upper bound
$$
 \left| \int \varphi_{j}(x) f(x) dx-\frac{1}{n}\sum_{i=1}^{n}\varphi_{j}(Z_{i})
                   \right|.
$$
We obtain
$$
  \mathbb{P} \left( \left|\frac{1}{n}\sum_{i=1}^{n}W^{(j)}_{i}\right|
        \geq \frac{t}{\sqrt{n}} \right)   \leq \psi(t) = 2 \exp(-\frac{t^{2}}{2B^{2}}) .
$$
So we can apply Theorem \ref{mainresult}.

\begin{cor}
In the context of density estimation, under Assumption ${\bf A}(\kappa)$, if the
$Z_{i}$ are iid with
density $f$ and if $\|\varphi_{j}\|_{\infty}<B$ for any $j\in\{1,\ldots,p\}$,
the choice $\lambda=4B\sqrt{2n\log(2p/\varepsilon)}$ leads to
$$
\mathbb{P}\left( \int
\left(f_{\hat{\theta}_{\lambda}}(x)-f_{\overline{\theta}}(x)\right)^{2}dx
                   \leq \frac{128 B^{2} }{\kappa}
                      \frac{\|\overline{\theta}\|_{0} \log\frac{2p}{\varepsilon}}
                               {n} \right)\geq 1-\varepsilon.
$$
\end{cor}
This result is essentially known, see \cite{spades}.

\subsection{Density estimation in the dependent case}

Note that if  as previously we work with bounded $\varphi_{j}(\cdot)$, we automatically
have moments of any order. So we will only state a result based on exponential
inequality.

So, using Theorem \ref{mainresult} and Theorem \ref{T2.1} we obtain:

\begin{cor}
 Let us assume that there are $L>0$ and $B\geq 1$ such that
$\varphi_{j}(\cdot)$ is $L$-Lipschitz and
$\|\varphi_{j}\|_{\infty}<B$ for any $j\in\{1,\ldots,p\}$. 
 Let us assume that $Z_{1}$, \ldots, $Z_{n}$ satisfy
$$ \forall k\geq 0,\quad
\sum_{s=0}^{\infty}(s+1)^{k} \eta_{Z}(s) \leq L_1 L_2^k (k!)^{\mu} $$
for some $L_{1},L_{2},\mu>0$.
Let us put a $c>0$, define
$$\mathcal{C}:=4BLL_{1} + (2^{3+\mu}B L_{1})^{1/(\mu+2)} c$$
and assume that $p$, $n$ and the confidence level $\varepsilon$ are such that
$$ p \leq \frac{\varepsilon}{2} \exp\left(\frac{c^{2}
n^{\frac{1}{\mu+2}}}{\mathcal{C}}\right).$$
Then
\begin{equation}
\label{eqdendep}
\mathbb{P}\left(
\int
\left(f_{\hat{\theta}_{\lambda}}(x)-f_{\overline{\theta}}(x)\right)^{2}dx
                   \leq
\frac{64\, \mathcal{C}  }{\kappa}
\frac{\|\overline{\theta}\|_{0}\log\left(\frac{2p}{\varepsilon}\right)}{n}
\right) \geq 1-\varepsilon.
\end{equation}
\end{cor}

\begin{rmk}
The assumption that the $\varphi_{j}$ are all $L$-Lipschitz for a constant $L$
excludes a lot of interesting dictionaries. If we assume that the $\varphi_{j}$
are $L(n)$-Lipschitz (this would be the case if we used the first $n$ functions
in the Fourier basis for example), then we will suffer a loss in \eqref{eqdendep}
when compared to the iid case.
However, note that Equation~\eqref{lipchchch}
below is the starting point of our proof, so we cannot hope to find a simple way to
remove this hypothesis when using $\eta$-weak dependence. This will be
the object of a future work.
\end{rmk}

\begin{proof}
As $\varphi_j$ is $K$-Lipschitz, using Proposition \ref{lipeta} we have:
\begin{equation} \eta_{\varphi_{j}(Z)}(r) \leq L\eta_{Z}(r) .
\label{lipchchch}
\end{equation}
So we have
$$  \forall k\geq 0,\quad
\sum_{k=1}^{\infty}(s+1)^{k} \eta_{\varphi_{j}(Z)}(r) \leq L L_1 L_2^k (k!)^{\mu}. $$
Moreover, following Remark~\ref{remarqueDN},
$$
\left| \mbox{cov }\left( \varphi_{j}(Z_{s_1})\cdots \varphi_{j}(Z_{s_u}),
\varphi_{j}(Z_{t_1})\cdots
\varphi_{j}(Z_{t_v}) \right) \right|
\leq  B^{u+v-1}  (u+v) L\cdot \eta_{Z}(r).
$$
So we can apply Theorem \ref{T2.1} with $\Psi(u,v)=u+v$ and we obtain
$$
\mathbb{P}\left(
\left|\sum_{i=1}^{n}W_{i}^{(j)}\right|
> t
\right)
\leq
2 \exp\left(\frac{-t^{2}/2}{A_{n} + B_{n}^{\frac1{\mu+2}}t^{\frac{2\mu+3}{\mu+2}}}\right)
$$
with $A_{n}=4nBLL_{1}$ and
$$B_{n}=2^{3+\mu} B L_{1} ,$$
in other words
$$
\mathbb{P}\left(
\left|\sum_{i=1}^{n}W_{i}^{(j)}\right|
> t
\right)
\leq
2 \exp\left(\frac{-t^{2}/2}{4nBLL_{1} + (2^{3+\mu}B L_{1})^{\frac1{\mu+2}}
 t^{\frac{2\mu+3}{\mu+2}}}\right).
$$
We then put $u=t\sqrt{n}$ to obtain
$$
\mathbb{P}\left(
\left|\frac{1}{n}\sum_{i=1}^{n}W_{i}^{(j)}\right|
> \frac{u}{\sqrt{n}}
\right)
\leq
2 \exp\left(\frac{-nu^{2}/2}{4BLL_{1} + (2^{3+\mu}B L_{1})^{\frac1{\mu+2}}
 n^{-\frac1{2\mu+4}} u^{\frac{2\mu+3}{\mu+2}}}\right).
$$
Here again, if we have
$$ u\leq c n^{1/(2\mu+4)}$$
then
\begin{multline*}
\mathbb{P}\left(
\left|\frac{1}{n}\sum_{i=1}^{n}W_{i}^{(j)}\right|
> \frac{u}{\sqrt{n}}
\right)
\\
\leq
2 \exp\left(\frac{-nu^{2}/2}{4BLL_{1} + (2^{3+\mu}B L_{1})^{\frac1{\mu+2}} c} \right)
2\exp\left(-\frac{n u^{2}}{2\,\mathcal{C}}\right)
=:\psi(u).
\end{multline*}
So we take, following Theorem \ref{mainresult},
$$ \lambda = \frac{4}{\sqrt{n}}\psi^{-1}\left(\frac{\varepsilon}{p}\right)
= 4\sqrt{\frac{\mathcal{C}
\log\left(\frac{2p}{\varepsilon}\right)}{n}}
$$
and we obtain, with probability at least $1-\varepsilon$,
$$
\int
\left(f_{\hat{\theta}_{\lambda}}(x)-f_{\overline{\theta}}(x)\right)^{2}dx
                   \leq
\frac{64 \,\mathcal{C}  }{\kappa}
\frac{\|\overline{\theta}\|_{0}\log\left(\frac{2p}{\varepsilon}\right)}{n}.
$$
\end{proof}

\section{Conclusion}

In this paper, we showed how the LASSO and other $\ell_{1}$-penalized
methods can be extended to the case of dependent random variables.

 An open and ambitious
question to be adressed later is  to find a good data-driven way to calibrate the regularization
parameter $\lambda$ when we don't know in advance the dependence coefficients
of our observations.

Anyway this first step with sparsity in the dependent setting is done for accurate applications and our brief simulations let us think that such techniques are reasonable for time series.

Here again extensions to random fields or to dependent point processes seem plausible.

\section{Proofs}
\label{proofs}

\begin{proof}[Proof of Theorem \ref{mainresult}]
By definition,
$$
\frac{1}{n}\sum_{i=1}^{n} Q(Z_{i},\hat{\theta}_{\lambda}) + \lambda
\|\hat{\theta}_{\lambda}\|_{1}
\leq
\frac{1}{n}\sum_{i=1}^{n} Q(Z_{i},\overline{\theta}) + \lambda
\|\overline{\theta}\|_{1}
$$
and so
\begin{multline}
\label{step1}
R(\hat{\theta}_{\lambda})-R(\overline{\theta})
\leq
 \int_{\mathcal{Z}^{n}}\frac{1}{n}\left\{\sum_{i=1}^{n}
\left[Q(z_{i},\hat{\theta}_{\lambda})
                -Q(z_{i},\overline{\theta})\right]\right\}
d\mathbb{P}(z_{1},\ldots,z_{n})
\\
   - \frac{1}{n}\sum_{i=1}^{n}
\left[Q(Z_{i},\hat{\theta}_{\lambda})-Q(Z_{i},\overline{\theta})\right]
+ \lambda \left(\|\overline{\theta}\|_{1}-\|\hat{\theta}_{\lambda}\|_{1}\right).
\end{multline}
Now, as $Q$ is quadratic wrt $\theta$ we have, for any $z$,
\begin{equation}
\label{help1}
 Q(z,\hat{\theta}_{\lambda}) = Q(z,\overline{\theta})
         +(\hat{\theta}_{\lambda}-\overline{\theta})'\frac{\partial
Q(z,\overline{\theta})}{\partial \theta}
         + \frac{1}{2}(\hat{\theta}_{\lambda}-\overline{\theta})'M
               (\hat{\theta}_{\lambda}-\overline{\theta}).
\end{equation}
Moreover, as $\overline{\theta}$ is the minimizer of $R(.)$, we have the
relation
\begin{equation}
\label{help2}
  \frac{\partial R(\overline{\theta})}{\partial \theta} =
  \int_{\mathcal{Z}^{n}} \frac{1}{n} \sum_{i=1}^{n}
          \frac{\partial Q(z_{i},\overline{\theta})}{\partial \theta}
d\mathbb{P}(z_{1},\ldots,z_{n})
             = 0.
\end{equation}
Pluging \eqref{help1} and \eqref{help2} into \eqref{step1} leads to
$$
R(\hat{\theta}_{\lambda})-R(\overline{\theta})
\leq
   (\hat{\theta}_{\lambda}-\overline{\theta})'\frac{1}{n}\sum_{i=1}^{n}
      \frac{\partial Q(Z_i,\overline{\theta})}{\partial \theta}
 + \lambda \left(\|\overline{\theta}\|_{1}-\|\hat{\theta}_{\lambda}\|_{1}\right)
$$
and then
\begin{equation}
\label{step2}
R(\hat{\theta}_{\lambda})-R(\overline{\theta})
\leq
   \|\hat{\theta}_{\lambda}-\overline{\theta}\|_{1}
          \sup_{1\leq j \leq p} \left|\frac{1}{n}\sum_{i=1}^{n}
               \frac{\partial Q(Z_i,\overline{\theta})}{\partial
\theta_{j}}\right|
 + \lambda
\left(\|\overline{\theta}\|_{1}-\|\hat{\theta}_{\lambda}\|_{1}\right).
\end{equation}
Now, we remind that we have the hypothesis
$$ \forall j\in\{1,\ldots,p\},\quad \mathbb{P}
    \left(\left|\frac{1}{n}\sum_{i=1}^{n}\frac{1}{2}\frac{\partial
Q(Z_{i},\overline{\theta})}
                      {\partial \theta_{j}}\right|\geq n^{\alpha-\frac{1}{2}} t \right)
\leq \psi(t)$$
that becomes, with a simple union bound argument,
$$ \mathbb{P}
    \left(\sup_{1\leq j \leq
p}\left|\frac{1}{n}\sum_{i=1}^{n}\frac{1}{2}\frac{\partial
Q(Z_{i},\overline{\theta})}
                      {\partial \theta_{j}}\right|\geq n^{\alpha-\frac{1}{2}} t \right)
\leq p \psi(t) $$
and so, if we put $t =\psi^{-1}(\varepsilon/p)$,
$$ \mathbb{P}
    \left(\sup_{1\leq j \leq
p}\left|\frac{1}{n}\sum_{i=1}^{n}\frac{1}{2}\frac{\partial
Q(Z_{i},\overline{\theta})}
                      {\partial \theta_{j}}\right|\geq n^{\alpha-\frac{1}{2}} \psi^{-1}
                            \left(\frac{\varepsilon}{p}\right) \right) \leq
\varepsilon. $$
Also remark that $n^{\alpha-1/2} \psi^{-1}(\varepsilon/p)=\lambda^{*}/4\leq \lambda/4$.
So until the end of the proof, we will work on the event
$$\left\{\omega\in\Omega:\sup_{1\leq j \leq
p}\left|\frac{1}{n}\sum_{i=1}^{n}\frac{1}{2}\frac{\partial
                              Q(Z_{i}(\omega),\overline{\theta})}
                      {\partial \theta_{j}}\right| \leq \frac{\lambda}{4} \right\} $$
true with probability at least $1-\varepsilon$. Going back to \eqref{step2}, we
have
$$
R(\hat{\theta}_{\lambda})-R(\overline{\theta})
\leq
    \frac{\lambda}{2}
       \|\hat{\theta}_{\lambda}-\overline{\theta}\|_{1}
 + \lambda \left(\|\overline{\theta}\|_{1}-\|\hat{\theta}_{\lambda}\|_{1}\right)
$$
and then
\begin{align*}
R(\hat{\theta}_{\lambda})-R(\overline{\theta}) + \frac{\lambda}{2}
       \|\hat{\theta}_{\lambda}-\overline{\theta}\|_{1}
&
\leq  \lambda \left(\|\hat{\theta}_{\lambda}-\overline{\theta}\|_{1}
       + \|\overline{\theta}\|_{1}-\|\hat{\theta}_{\lambda}\|_{1}\right)
\\[-1pt]
& = \lambda \left( \sum_{j=1}^{p}
|(\hat{\theta}_{\lambda})_{j}-\overline{\theta}_{j}|
                +
\sum_{j=1}^{p}(|\overline{\theta}_{j}|-|(\hat{\theta}_{\lambda})_{j}|)
          \right)
\\[-1pt]
& = \lambda \left( \sum_{j:\overline{\theta}_{j}\neq 0}
|(\hat{\theta}_{\lambda})_{j}-\overline{\theta}_{j}|
                + \sum_{j:\overline{\theta}_{j}\neq
0}(|\overline{\theta}_{j}|-|(\hat{\theta}_{\lambda})_{j}|)
          \right)
\end{align*}
that leads to the following inequality that will play a central role in the end
of the proof:
\begin{equation}
\label{step3}
R(\hat{\theta}_{\lambda})-R(\overline{\theta}) + \frac{\lambda}{2}
       \|\hat{\theta}_{\lambda}-\overline{\theta}\|_{1}
\leq
2\lambda \sum_{j:\overline{\theta}_{j}\neq 0}
|(\hat{\theta}_{\lambda})_{j}-\overline{\theta}_{j}|.
\end{equation}
First, if we remind that $R(\hat{\theta}_{\lambda})-R(\overline{\theta})\geq 0$,
\eqref{step3} leads to
$$
\|\hat{\theta}_{\lambda}-\overline{\theta}\|_{1}
\leq 4 \sum_{j:\overline{\theta}_{j}\neq 0}
|(\hat{\theta}_{\lambda})_{j}-\overline{\theta}_{j}|
$$
and so
$$
\sum_{j:\overline{\theta}_{j} = 0}
|(\hat{\theta}_{\lambda})_{j}-\overline{\theta}_{j}|
\leq 3 \sum_{j:\overline{\theta}_{j}\neq 0}
|(\hat{\theta}_{\lambda})_{j}-\overline{\theta}_{j}|.
$$
So we can take $v:=\hat{\theta}_{\lambda}-\overline{\theta}$ in Assumption ${\bf
A}(\kappa)$.
So, \eqref{step3} leads to
\begin{align}
R(\hat{\theta}_{\lambda})-R(\overline{\theta}) +
\frac{\lambda}{2}\|\hat{\theta}_{\lambda}-\overline{\theta}\|_{1}
\label{step4}
& \leq 2\lambda \sum_{j:\overline{\theta}_{j}\neq 0}
|(\hat{\theta}_{\lambda})_{j}-\overline{\theta}_{j}|
\\[-1pt]
\nonumber
& \leq 2 \lambda \left( \|\overline{\theta}\|_{0}
                     \sum_{j:\overline{\theta}_{j}\neq 0}
[(\hat{\theta}_{\lambda})_{j}-\overline{\theta}_{j}]^{2}
                           \right)^{\frac{1}{2}}
\\[-1pt]
\nonumber
& \leq 2 \lambda \left( \frac{\|\overline{\theta}\|_{0}}{\kappa}
(\hat{\theta}_{\lambda}-\overline{\theta})'\frac{\mathbf{M}}{2}(\hat{\theta}_{
\lambda}-\overline{\theta})
                    \right)^{\frac{1}{2}}
\\[-1pt]
& = 2 \lambda \left( \frac{\|\overline{\theta}\|_{0}}{\kappa}
                    \left[R(\hat{\theta}_{\lambda})-R(\overline{\theta})\right]
                    \right)^{\frac{1}{2}}.
\label{step5}
\end{align}
We conclude that
$$
R(\hat{\theta}_{\lambda})-R(\overline{\theta})
\leq \frac{4\lambda^{2} \|\overline{\theta}\|_{0}}{\kappa}.
$$
Now remark that \eqref{step4} to \eqref{step5} states that a convex quadratic
function of $\bigl[R(\hat{\theta}_{\lambda})-R(\overline{\theta})\bigr]$
is negative, so both roots of that quadratic are real. This leads to
$$ \|\hat{\theta}_{\lambda}-\overline{\theta}\|_{1}
     \leq \frac{2\lambda \|\overline{\theta}\|_{0}}{\kappa}.$$
This ends the proof.
\end{proof}
\vspace*{-6pt}
\eject

\begin{proof}[Proof of Proposition \ref{momcomb}]
First
$$
\left|\E(W^{(j)}_1+\cdots+W^{(j)}_n)^{\ell}\right|\le
\ell!A^{(j)}_{\ell,n}\equiv \ell!\sum_{1\le k_1,\ldots,k_{\ell}\le n}\left|\E
W^{(j)}_{k_{1}}\cdots W^{(j)}_{k_{\ell}}\right|.
$$
 The same combinatorial arguments as in \cite{MR1719345} yield for $p\le 2q$
\begin{eqnarray}
\label{eqrec}
A^{(j)}_{\ell,n}&\le&
C^{(j)}_{q,\ell,n}+\sum_{m=2}^{\ell-2}A^{(j)}_{m,n}A^{(j)}_{\ell-m,n}, \mbox{
where}\\
\label{coefsum}C^{(j)}_{q,\ell,n}&\equiv&(p-1)n\sum_{r=0}^{n-1}(r+1)^{\ell-2}c_{
W^{(j)},2q}(r).
\end{eqnarray}
Let us now assume the condition (\ref{borneC}) then
\begin{eqnarray*}
C^{(j)}_{q,\ell,n}&\le&C(\ell-1)n\sum_{r=0}^{n-1}(r+1)^{\ell-2-q}\\
&\le&C(\ell-1)n\int_{2}^{n+1}x^{\ell-2-q}\,dx, \qquad \mbox{ if }  \ell\ge q+2\\
&\le& C(q+1)(n+1)^{2} , \qquad  \qquad  \qquad \mbox{ if }  \ell= q+2\\
&\le& C\frac{\ell-1}{\ell-q-2}(n+1)^{\ell-q} , \ \qquad  \qquad \mbox{ if }
\ell\ge q+2\\
&\le&C(\ell-1)n\int_{1}^{n}x^{\ell-2-q}\,dx, \qquad \mbox{ if } \ell<q+2\\
&\le& C\frac{\ell-1}{q+2-\ell}n. \end{eqnarray*}
A rough bound is thus $C^{(j)}_{q,\ell,n}\le C(\ell-1)n^{(\ell-q)\vee1} $ and we
thus derive
\begin{equation}
\label{eq7}
\begin{array}{lllllllll}
A^{(j)}_{2,n}&\le& Cn,&
A^{(j)}_{3,n}&\le& 2Cn,&
A^{(j)}_{4,n}&\le& 4Cn^2\\
A^{(j)}_{5,n}&\le& 8Cn^2,&
A^{(j)}_{6,n}&\le& 17Cn^3.
\end{array}
\end{equation}
Now using precisely condition (\ref{borneC}) with the relation (\ref{eqrec}) we
see that
if $a_2=1$, and $a_3=2$ then the sequence recursively defined as
\begin{equation}
\label{recur}
a_m=m-1+\sum_{k=2}^{m-2}a_k a_{m-k}
\end{equation}
satisfies
$
A_m^{(j)}\le a_mC^{[\frac m2]} n^{[\frac m2]}.
$
Remember that
\begin{equation*}
d_m\equiv \frac1m\frac{(2m-2)!}{((m-1)!)^2}, \qquad m=2, 3, \ldots\end{equation*}
hence as in \cite{MR1719345} we quote that $$a_{m}\le d_m $$ is less that the
$m$-th Catalan number, $d_m$ and this ends the proof.
\end{proof}

\makeatletter

\def\@lbibitem[#1]#2{%
  \if\relax\@extra@b@citeb\relax\else
    \@ifundefined{br@#2\@extra@b@citeb}{}{%
     \@namedef{br@#2}{\@nameuse{br@#2\@extra@b@citeb}}}\fi
   \@ifundefined{b@#2\@extra@b@citeb}{\def\NAT@num{}}{\NAT@parse{#2}}%
   \item[\hyper@natanchorstart{#2\@extra@b@citeb}\@biblabel{\NAT@num}%
    \hyper@natanchorend\hfill]%
    \NAT@ifcmd#1(@)(@)\@nil{#2}}

\makeatother

\end{document}